\documentclass[11pt]{article}

\usepackage[T1]{fontenc}
\usepackage{amsmath,amssymb,amsthm}
\usepackage{graphicx}
\usepackage{hyperref}
\usepackage{url}
\usepackage[margin=1in]{geometry}

\newtheorem{lemma}{Lemma}
\newtheorem{proposition}{Proposition}
\newtheorem{theorem}{Theorem}
\theoremstyle{remark}
\newtheorem{remark}{Remark}

\title{Hodge Coercivity and Global Dynamics in Two-Field Edge-Cochain Systems with MHD-Type Cancellation}
\author{Moses Boudourides\\
School of Professional Studies, Northwestern University\\
\texttt{Moses.Boudourides@northwestern.edu}\\
ORCID: \href{https://orcid.org/0000-0002-6157-5647}{0000-0002-6157-5647}}
\date{}

\begin{document}

\maketitle

\noindent\textbf{Running title:} Hodge Coercivity and Dynamics

\begin{abstract}
A finite-dimensional two-field system for divergence-free edge cochains is introduced. Its MHD-type designation refers only to a quadratic exchange pattern and exact total-energy cancellation; it is not a physical MHD discretization. A general cancellation class is separated from a corrected explicit realization: the anticommutator $D(a)J+JD(a)$ is skew-symmetric for diagonal $D(a)$ and skew-symmetric $J$, and its projected bilinear map has the required trilinear antisymmetry. The central result is a Hodge coercivity criterion: the full divergence-free space admits the Poincar\'e-type estimate needed for dissipativity if and only if its harmonic $1$-cochain space is trivial. Under this condition, global existence, an exact energy identity, an absorbing ball, and a compact global attractor follow. When harmonic modes are present, a harmonic-decoupled interaction class yields invariant harmonic affine fibres and fibre-wise attractors. Deterministic disk, annular, and two-hole examples illustrate the spectral criterion, energy law, and distinction between general harmonic exchange and harmonic-fibre invariance.
\end{abstract}


\section{Introduction and contribution hierarchy}

Discrete Hodge theory supplies a natural language for edge flows, cochains, incidence operators, and higher-order Laplacians. Foundational treatments connect the cochain complex to Hodge decompositions and graph Laplacians \cite{Lim2020,HorakJost2013}, while recent work studies random walks, consensus, and higher-order organization on simplicial complexes \cite{SchaubEtAl2020,ZieglerEtAl2022,BensonGleichLeskovec2016,BattistonEtAl2021}. These developments complement finite element exterior calculus (FEEC), where differential complexes and compatible operators guide stable discretizations \cite{ArnoldFalkWinther2006,ArnoldFalkWinther2010}. The present paper considers a different question: which long-time consequences follow for a constrained two-field edge-cochain system on a cell complex when Hodge-Laplacian dissipation is combined with an algebraic quadratic exchange law?

The system is termed \emph{MHD-type} for a limited and explicit reason. In incompressible MHD, the schematic quadratic terms have the exchange form
\[
(u\cdot\nabla)u-(h\cdot\nabla)h,
\qquad
(u\cdot\nabla)h-(h\cdot\nabla)u.
\]
The cochain system below retains the corresponding two-field exchange pattern and total-energy cancellation. It does not retain the differential transport operator, a Lorentz-force derivation, a discrete induction law, physical helicities, continuum consistency, or mesh-convergence properties. Structure-preserving MHD methods have substantially stronger objectives; examples include the compatible finite element constructions of Hu, Ma, and Xu \cite{HuMaXu2017}, the FEEC--finite-volume splitting method of Hiptmair and Pagliantini \cite{HiptmairPagliantini2018}, and discrete Lie-advection methods for differential forms \cite{MullenEtAl2011}. The model here is an MHD-inspired edge-cochain dynamical system on a cell complex, not a numerical discretization of those physical equations.

The contribution has three layers. First, a general theorem is established for continuous bilinear interactions satisfying a trilinear cancellation identity. Second, a corrected anticommutator construction provides a nonempty, directly implementable class of such interactions. Third, and centrally, topology is shown to govern whether the ordinary full-space dissipative argument is available: the required Hodge coercivity inequality holds on the divergence-free edge space exactly when harmonic $1$-cochains are absent. A further harmonic-decoupled interaction class provides a nontrivial positive result in the noncoercive case by producing invariant affine fibres and fibre-wise attractors.

The long-time motivation is related to, but distinct from, attractor theory for magnetic PDEs. Boudourides and Nikoudes \cite{BoudouridesNikoudes1990} established a maximal attractor and dimension estimates for the two-dimensional magnetic B\'enard problem. Catania and Secchi \cite{CataniaSecchi2010} proved global existence and finite-dimensional attractor results for a three-dimensional double-viscous MHD-$\alpha$ model. In the present finite-dimensional setting, compact attractor existence after an absorbing estimate is standard. The mathematical interest lies instead in identifying exactly which Hodge-topological condition supplies that estimate and in describing one structured noncoercive alternative.

\subsection{Relation with graph-Hodge dynamics}

The graph-Hodge literature supplies several relevant, but distinct, points of comparison. A Hodge decomposition separates edge fields into gradient, curl, and harmonic parts; Lim \cite{Lim2020} emphasizes that this structure can be developed by linear algebra on incidence matrices, while Horak and Jost \cite{HorakJost2013} provide a spectral framework for combinatorial Laplacians on simplicial complexes. The present phase space starts after the gradient component has been removed: $\mathcal V=\ker d_0^*$ is the divergence-free edge space, and the remaining splitting is between a dissipative complement and harmonic circulation directions.

Many existing dynamical uses of the Hodge Laplacian are linear. The normalized $1$-Laplacian of Schaub et al. \cite{SchaubEtAl2020} governs edge-space random walks, and the balanced Hodge Laplacians of Ziegler, Hlinka, and Thurner \cite{ZieglerEtAl2022} quantify how lower- and higher-order interactions affect edge consensus. Higher-order network work more generally uses simplicial complexes to represent interactions that cannot be reduced to pairs \cite{BensonGleichLeskovec2016,BattistonEtAl2021}; nonlinear synchronization on simplicial complexes provides another distinct line of development \cite{MillanTorresBianconi2020}. The present problem differs in two respects. It is a nonlinear two-field dissipative evolution rather than a diffusion, consensus, or phase-synchronization equation, and the key question is not simply spectral convergence but the availability of a global energy estimate after the nonlinear exchange terms have cancelled. Thus the role of the Hodge Laplacian here is not only spectral smoothing or consensus formation, but the creation, or failure, of a coercive energy mechanism after nonlinear cancellation.

This distinction clarifies the novelty boundary. The paper does not introduce a new general theory of Hodge Laplacians, nor does it claim a physical transport discretization. It identifies a class of nonlinear constrained cochain systems for which the Hodge kernel has a direct dynamical consequence: it is exactly the obstruction to the standard full-phase-space absorbing-set proof. The harmonic-fibre theorem below then shows how a further algebraic condition can replace full-space coercivity by fibre-wise coercivity.

\section{Finite cochain complexes and Hodge coercivity}

Let $K=(K_0,K_1,K_2)$ be a finite oriented two-dimensional cell complex, where $K_0$, $K_1$, and $K_2$ denote its vertices, edges, and faces, respectively. The real cochain spaces $C^q(K)$, $q=0,1,2$, have fixed Euclidean inner products. The coboundary operators are
\[
d_0:C^0(K)\longrightarrow C^1(K),
\qquad
d_1:C^1(K)\longrightarrow C^2(K),
\qquad
d_1d_0=0.
\]
In oriented bases, $d_0$ is the edge--vertex incidence matrix and $d_1$ records oriented face boundaries. The divergence-free edge-cochain space is
\[
\mathcal V:=\ker d_0^*\subset C^1(K),
\]
and $P:C^1(K)\to\mathcal V$ denotes the Euclidean orthogonal projector. Equivalently,
\[
P=I-d_0(d_0^*d_0)^\dagger d_0^*,
\]
where the Moore--Penrose inverse accommodates the constant nullspace. Thus $P^*=P$, $P^2=P$, and $d_0^*P=0$.

The dissipative operator is the full $1$-cochain Hodge Laplacian
\[
\Delta_1:=d_0d_0^*+d_1^*d_1.
\]
For $v\in\mathcal V$,
\begin{equation}
\langle\Delta_1v,v\rangle=\|d_1v\|^2.
\label{eq:hodgeenergy}
\end{equation}
The down-Laplacian $d_0d_0^*$ alone vanishes on $\mathcal V$ and therefore cannot provide the dissipation used below.

\begin{lemma}[Invariance and commutation]\label{lem:invariance}
The space $\mathcal V$ is invariant under $\Delta_1$, and $P\Delta_1=\Delta_1P$.
\end{lemma}

\begin{proof}
For $v\in\mathcal V$,
\[
d_0^*\Delta_1v=d_0^*d_0d_0^*v+d_0^*d_1^*d_1v=0+(d_1d_0)^*d_1v=0.
\]
Thus $\Delta_1\mathcal V\subset\mathcal V$. Self-adjointness then implies invariance of $\mathcal V^\perp$, and the orthogonal decomposition yields the commutation relation.
\end{proof}

Define the harmonic subspace and its complementary divergence-free subspace by
\begin{equation}
\mathcal H^1:=\ker\Delta_1\cap\mathcal V=\ker d_1\cap\ker d_0^*,
\qquad
\mathcal V_0:=\mathcal V\cap(\mathcal H^1)^\perp.
\label{eq:hodgedecomp}
\end{equation}
Let $Q$ be the orthogonal projector from $\mathcal V$ onto $\mathcal H^1$ and put $P_0:=I-Q$ on $\mathcal V$.

\begin{theorem}[Hodge coercivity criterion]\label{thm:coercivity}
The following are equivalent:
\begin{enumerate}
\item $\mathcal H^1=\{0\}$;
\item there exists $\lambda_*>0$ such that
\begin{equation}
\langle\Delta_1v,v\rangle\geq\lambda_*\|v\|^2
\qquad(v\in\mathcal V).
\label{eq:coercivity}
\end{equation}
\end{enumerate}
When $\mathcal H^1\neq\{0\}$, the analogous estimate holds on $\mathcal V_0$ with the smallest positive eigenvalue $\lambda_+$ of $\Delta_1|_{\mathcal V}$.
\end{theorem}

\begin{proof}
The restricted operator $\Delta_1|_{\mathcal V}$ is self-adjoint and nonnegative, and its kernel is $\mathcal H^1$. If the kernel is trivial, its finite spectrum is strictly positive and its minimum is $\lambda_*$. Conversely, a nonzero harmonic vector contradicts \eqref{eq:coercivity}. Removing the zero eigenspace gives the statement on $\mathcal V_0$.
\end{proof}

Theorem~\ref{thm:coercivity} is the topological pivot of the paper. Harmonic $1$-cochains are undamped by the Hodge Laplacian, so their presence prevents a Poincar\'e inequality on the full divergence-free space. The numerical examples below make this obstruction visible through restricted spectra.

\subsection{Phase-space geometry and modal coordinates}

The decomposition in \eqref{eq:hodgedecomp} is orthogonal and preserved by the linear Hodge semigroup. If $\{\phi_j\}$ is an orthonormal eigenbasis of $\Delta_1|_{\mathcal V}$ with eigenvalues $\lambda_j\geq0$, then the harmonic vectors are precisely the modes with $\lambda_j=0$, and the remaining modes satisfy $\lambda_j\geq\lambda_+$. In the linear, unforced limit, the components of a solution satisfy
\[
u_j(t)=e^{-\nu\lambda_jt}u_j(0),
\qquad
h_j(t)=e^{-\eta\lambda_jt}h_j(0).
\]
Consequently, the Hodge Laplacian alone leaves $Qu$ and $Qh$ unchanged and exponentially damps the complementary components. This simple modal calibration is useful because it separates two issues that are sometimes conflated: a nontrivial harmonic space is a spectral obstruction to full-space coercivity, whereas its actual nonlinear evolution depends on the interaction map.

The restricted positive eigenvalue has a direct variational characterization,
\[
\lambda_+=\min_{0\neq v\in\mathcal V_0}
\frac{\langle\Delta_1v,v\rangle}{\|v\|^2}.
\]
Thus $\lambda_+$ is not merely a numerical quantity extracted from a matrix; it is the exact constant that controls every complementary-space energy estimate. The disk, annulus, and two-hole examples below have harmonic dimensions zero, one, and two, respectively. Their spectra therefore provide a direct computational realization of the alternative in Theorem~\ref{thm:coercivity}.

\section{MHD-type edge-cochain dynamics}

Let $\nu>0$ and $\eta>0$. The unknowns are $u(t),h(t)\in\mathcal V$, interpreted as velocity-like and magnetic-like edge cochains. Let $f,g\in C^1(K)$ be time-independent forcings. For a continuous bilinear map
\[
B:\mathcal V\times\mathcal V\longrightarrow\mathcal V,
\]
write
\[
b(a,b,c):=\langle B(a,b),c\rangle.
\]
The structural assumption is the trilinear antisymmetry
\begin{equation}
b(a,b,c)=-b(a,c,b)
\qquad(a,b,c\in\mathcal V).
\label{eq:antisymmetry}
\end{equation}
In particular, $b(a,b,b)=0$. Consider
\begin{align}
\dot u&=-\nu\Delta_1u+B(u,u)-B(h,h)+Pf,\label{eq:u}\\
\dot h&=-\eta\Delta_1h+B(u,h)-B(h,u)+Pg.\label{eq:h}
\end{align}
Lemma~\ref{lem:invariance} and the range condition on $B$ show that $\mathcal V\times\mathcal V$ is invariant.

\subsection{A corrected explicit realization}

\begin{proposition}[Corrected anticommutator realization]\label{prop:toy}
Choose an oriented edge basis of $C^1(K)\simeq\mathbb R^m$, fix a skew-symmetric matrix $J\in\mathbb R^{m\times m}$, and let $D(a)=\operatorname{diag}(a_1,\ldots,a_m)$. Define
\begin{equation}
M(a):=D(a)J+JD(a),
\qquad
B(a,b):=PM(a)b
\qquad(a,b\in\mathcal V).
\label{eq:anticommutator}
\end{equation}
Then $B(a,b)\in\mathcal V$, $M(a)^\top=-M(a)$, and $B$ satisfies \eqref{eq:antisymmetry}. Moreover,
\begin{equation}
\|B(a,b)\|\leq2\|J\|\,\|a\|\,\|b\|.
\label{eq:bilinearbound}
\end{equation}
\end{proposition}

\begin{proof}
The range statement follows immediately from the leading projector $P$. Since $D(a)^\top=D(a)$ and $J^\top=-J$,
\[
M(a)^\top=J^\top D(a)^\top+D(a)^\top J^\top=-JD(a)-D(a)J=-M(a).
\]
For $c\in\mathcal V$, $Pc=c$ and self-adjointness of $P$ give
\[
b(a,b,c)=\langle PM(a)b,c\rangle
=\langle M(a)b,c\rangle
=-\langle b,M(a)c\rangle
=-b(a,c,b).
\]
Finally, $\|P\|=1$ and $\|D(a)\|\leq\|a\|_\infty\leq\|a\|$ yield \eqref{eq:bilinearbound}.
\end{proof}

\begin{remark}[Why the anticommutator is essential]
For symmetric $D(a)$ and skew-symmetric $J$, the ordinary commutator $D(a)J-JD(a)$ is symmetric. It therefore cannot yield \eqref{eq:antisymmetry} through the preceding argument. The anticommutator in \eqref{eq:anticommutator} is the only toy interaction used in the deterministic computations.
\end{remark}

\begin{remark}[Scope of the realization]
The matrix $J$ can couple nonadjacent edge coordinates and the projector $P$ is generally global. Thus Proposition~\ref{prop:toy} demonstrates an explicit admissible realization of the cancellation class; it does not derive a local cochain transport operator. Constructing local interactions from cup products, discrete Lie derivatives, or compatible MHD complexes remains a separate problem \cite{MullenEtAl2011,HiptmairPagliantini2018}.
\end{remark}

\subsection{Interaction class and exchange symmetries}

The energy argument uses only the range condition and \eqref{eq:antisymmetry}; it does not use a coordinate representation of $B$. This separation is important. Let
\[
\mathcal N(u,h):=\bigl(B(u,u)-B(h,h),\,B(u,h)-B(h,u)\bigr).
\]
Then the cancellation identity is equivalently
\begin{equation}
\langle\mathcal N(u,h),(u,h)\rangle_{\mathcal V\times\mathcal V}=0.
\label{eq:vectorcancellation}
\end{equation}
The nonlinear vector field is therefore energy-neutral, even though the individual field equations exchange energy through both self- and cross-couplings. This is the finite-dimensional analogue of the total-energy cancellation that motivates the MHD-type terminology.

The cancellation class is stable under two useful operations. If $B_1$ and $B_2$ satisfy \eqref{eq:antisymmetry}, then every linear combination $\alpha B_1+\gamma B_2$ does as well. Further, let $R$ be an orthogonal projector on $\mathcal V$ and let $\widetilde B$ be admissible. On the restricted phase space $R\mathcal V$, define $B_R(a,b):=R\widetilde B(a,b)$ for $a,b\in R\mathcal V$. For $c\in R\mathcal V$, orthogonality gives $\langle B_R(a,b),c\rangle=\langle\widetilde B(a,b),c\rangle=-\langle\widetilde B(a,c),b\rangle=-\langle B_R(a,c),b\rangle$. Thus $B_R$ satisfies the same trilinear antisymmetry on $R\mathcal V$. The harmonic-decoupled construction in Section~\ref{sec:topology} is a further projected construction that also removes harmonic dependence. These elementary closure properties show that the cancellation assumption specifies an interaction class rather than a single engineered matrix formula.

For the anticommutator realization, bilinearity also gives an explicit local Lipschitz estimate. If $a,b,c,d\in\mathcal V$, then
\[
B(a,b)-B(c,d)=B(a-c,b)+B(c,b-d),
\]
and therefore
\begin{equation}
\|B(a,b)-B(c,d)\|
\leq2\|J\|\bigl(\|a-c\|\,\|b\|+\|c\|\,\|b-d\|\bigr).
\label{eq:localLipschitz}
\end{equation}
The estimate makes the continuous-semiflow formulation below quantitative on every bounded set. It is not, by itself, a derivative-of-semiflow estimate sufficient for a nontrivial attractor-dimension bound.

\subsection{Global solvability and a continuous semiflow}

For the realization in Proposition~\ref{prop:toy}, the right-hand side of \eqref{eq:u}--\eqref{eq:h} is polynomial and locally Lipschitz. The cancellation identity prevents finite-time blow-up even when the complex is not coercive.

\begin{lemma}[Global existence]\label{lem:global}
For every initial datum in $\mathcal V\times\mathcal V$ and fixed $f,g$, system \eqref{eq:u}--\eqref{eq:h} has a unique global solution.
\end{lemma}

\begin{proof}
The energy calculation below gives
\[
\frac12E'(t)+\nu\|\Delta_1^{1/2}u\|^2+\eta\|\Delta_1^{1/2}h\|^2
=\langle Pf,u\rangle+\langle Pg,h\rangle,
\]
where $E=\|u\|^2+\|h\|^2$. Dropping nonnegative dissipation gives $E'(t)\leq2F\sqrt{E(t)}$, where $F^2=\|Pf\|^2+\|Pg\|^2$. Hence $\sqrt{E(t)}\leq\sqrt{E(0)}+Ft$ on every existence interval, and finite-dimensional continuation completes the proof.
\end{proof}

The solution map $S(t)$ is therefore a continuous semiflow on $\mathcal V\times\mathcal V$. This elementary observation is used below only to apply the standard absorbing-set construction; no nontrivial attractor-dimension estimate is claimed.

\subsection{Continuous dependence on bounded sets}

For the explicit realization, continuous dependence can be quantified. Let $(u,h)$ and $(\bar u,\bar h)$ be two solutions with the same forcing, and assume that all four component norms remain bounded by $R$ on $[0,T]$. Bilinearity and \eqref{eq:bilinearbound} imply
\[
\|B(u,u)-B(\bar u,\bar u)\|
\leq4\|J\|R\|u-\bar u\|,
\]
and the same estimate holds for each mixed difference after replacing one factor by its difference. Consequently, if
\[
\mathcal E(t):=\|u(t)-\bar u(t)\|^2+\|h(t)-\bar h(t)\|^2,
\]
there is a constant $C_J>0$, depending only on $J$, such that
\begin{equation}
\mathcal E'(t)\leq C_JR\,\mathcal E(t)
\qquad(0\leq t\leq T).
\label{eq:continuousdependence}
\end{equation}
A direct Euclidean-norm estimate may take $C_J=16\|J\|$. The dissipative terms have simply been dropped in this upper bound. Gr\"onwall's inequality yields
\[
\mathcal E(t)\leq e^{C_JRt}\mathcal E(0).
\]
Although \eqref{eq:continuousdependence} is not a contraction estimate, it makes explicit why the polynomial system defines a continuous semiflow and why the attractor construction can be based on a compact absorbing ball. Obtaining a useful volume-contraction estimate would require sharper information about the derivative of the nonlinear flow.

\section{Energy identity and coercive long-time dynamics}

Set
\[
E(t):=\|u(t)\|^2+\|h(t)\|^2,
\qquad
F^2:=\|Pf\|^2+\|Pg\|^2.
\]

\begin{theorem}[Exact energy identity]\label{thm:energy}
Every solution satisfying \eqref{eq:antisymmetry} obeys
\begin{equation}
\frac12\frac{dE}{dt}
+\nu\|\Delta_1^{1/2}u\|^2
+\eta\|\Delta_1^{1/2}h\|^2
=\langle Pf,u\rangle+\langle Pg,h\rangle.
\label{eq:energyidentity}
\end{equation}
\end{theorem}

\begin{proof}
Take inner products of \eqref{eq:u} and \eqref{eq:h} with $u$ and $h$, respectively. The nonlinear contribution is
\[
b(u,u,u)-b(h,h,u)+b(u,h,h)-b(h,u,h).
\]
The first and third terms vanish because $b(a,b,b)=0$, and the remaining two cancel by \eqref{eq:antisymmetry}.
\end{proof}

Assume first that $\mathcal H^1=\{0\}$ and define
\[
\beta:=\min(\nu,\eta)\lambda_*>0.
\]
Then the dissipative terms in \eqref{eq:energyidentity} are bounded below by $\beta E$. Young's inequality yields
\[
2\langle Pf,u\rangle+2\langle Pg,h\rangle
\leq\beta E+\frac{F^2}{\beta}.
\]

\begin{proposition}[Forced scalar inequality]\label{prop:scalar}
On a coercive phase space,
\begin{equation}
E'(t)\leq-\beta E(t)+\frac{F^2}{\beta},
\label{eq:scalarineq}
\end{equation}
and consequently
\begin{equation}
E(t)\leq E(0)e^{-\beta t}+\frac{F^2}{\beta^2}(1-e^{-\beta t}).
\label{eq:forcedbound}
\end{equation}
\end{proposition}

The rate in Proposition~\ref{prop:scalar} intentionally sacrifices a factor of two: Young's inequality consumes one of the two coercive contributions in order to absorb the forcing. When $f=g=0$, no Young estimate is used and the exact energy identity retains the sharper exponent $2\beta$.

\subsection{Forced balance and time-averaged dissipation}

The scalar estimate supplies more than boundedness. Integrating the exact identity gives, for every $T>0$,
\begin{align}
2\int_0^T\bigl(\nu\|\Delta_1^{1/2}u\|^2+\eta\|\Delta_1^{1/2}h\|^2\bigr)\,dt
&=E(0)-E(T)\\
&\quad+2\int_0^T\bigl(\langle Pf,u\rangle+\langle Pg,h\rangle\bigr)\,dt.
\label{eq:integratedbalance}
\end{align}
Thus a bounded forced trajectory has bounded time-averaged dissipation. More specifically, Proposition~\ref{prop:scalar} and Cauchy--Schwarz show that the mean dissipation is controlled entirely by the projected forcing magnitude and the absorbing radius. This observation is elementary in finite dimensions, but it explains why the numerical energy diagnostic is meaningful: the pointwise cancellation law is exact, while numerical quadrature only approximates the integral in \eqref{eq:integratedbalance}.

The forcing is deliberately time independent here because it generates an autonomous semiflow. Time-dependent forcing would lead to a nonautonomous process and requires a pullback- or uniform-attractor framework, which is outside the current scope. Likewise, a quantitative fractal-dimension estimate would require control of the linearized flow on the absorbing set. Neither follows from energy cancellation alone.

\begin{theorem}[Positively invariant absorbing ball and compact attractor]\label{thm:attractor}
Assume $\mathcal H^1=\{0\}$. For every $R>F/\beta$, the ball
\[
\mathcal B_R:=\{(u,h)\in\mathcal V\times\mathcal V:E\leq R^2\}
\]
is positively invariant and absorbing. The semiflow has the compact global attractor
\[
\mathcal A=\omega(\mathcal B_R):=\bigcap_{\tau\geq0}\overline{\bigcup_{t\geq\tau}S(t)\mathcal B_R}.
\]
\end{theorem}

\begin{proof}
On $E=R^2$, \eqref{eq:scalarineq} gives
\[
E'(t)\leq-\beta R^2+\frac{F^2}{\beta}<0,
\]
which proves positive invariance. Formula \eqref{eq:forcedbound} proves absorption. Lemma~\ref{lem:global} and local Lipschitz continuity give a continuous semiflow, while the closed bounded ball is compact in finite dimensions. The displayed closed omega-limit set is therefore nonempty, compact, invariant, and attracting \cite{Temam1997,ConstantinFoias1988}.
\end{proof}

\begin{theorem}[Unforced coercive decay]\label{thm:unforced}
If $f=g=0$ and $\mathcal H^1=\{0\}$, then
\[
E(t)\leq E(0)e^{-2\beta t}.
\]
Thus the origin is the global attractor.
\end{theorem}

\begin{proof}
With zero forcing, \eqref{eq:energyidentity} and coercivity yield $E'(t)\leq-2\beta E(t)$.
\end{proof}

\begin{remark}[Modesty of the attractor conclusion]
In finite dimensions, global existence plus a compact absorbing set implies a compact global attractor. Theorem~\ref{thm:attractor} is recorded for completeness, but the nonroutine issue is the topological hypothesis needed to obtain the absorbing estimate on the full divergence-free phase space.
\end{remark}

\section{Harmonic affine fibres in the noncoercive case}\label{sec:topology}

When $\mathcal H^1\neq\{0\}$, applying $Q$ to \eqref{eq:u}--\eqref{eq:h} gives
\begin{align}
\frac{d}{dt}Qu&=Q\bigl\{B(u,u)-B(h,h)\bigr\}+QPf,\label{eq:harmonicu}\\
\frac{d}{dt}Qh&=Q\bigl\{B(u,h)-B(h,u)\bigr\}+QPg.\label{eq:harmonich}
\end{align}
Harmonic modes are undamped by $\Delta_1$, but they need not be conserved under an arbitrary admissible interaction. Equations \eqref{eq:harmonicu}--\eqref{eq:harmonich} rule out any universal persistence statement based only on nontrivial cohomology.

A structured subclass yields a positive alternative. Let $\widetilde B$ be any bilinear interaction satisfying \eqref{eq:antisymmetry}, and define its harmonic-decoupled version by
\begin{equation}
B_0(a,b):=P_0\widetilde B(P_0a,P_0b).
\label{eq:decoupledB}
\end{equation}
The construction used in the deterministic two-hole experiment is obtained by taking $\widetilde B$ from Proposition~\ref{prop:toy}; explicitly,
\[
B_0(a,b)=P_0\bigl(D(P_0a)J+JD(P_0a)\bigr)P_0b.
\]

\begin{lemma}[Decoupled cancellation]\label{lem:decoupled}
The map $B_0$ satisfies \eqref{eq:antisymmetry} and $QB_0(a,b)=0$ for all $a,b\in\mathcal V$.
\end{lemma}

\begin{proof}
The second assertion follows from $QP_0=0$. For the first, orthogonality of $P_0$ and the antisymmetry of $\widetilde B$ give
\[
\langle B_0(a,b),c\rangle
=\langle\widetilde B(P_0a,P_0b),P_0c\rangle
=-\langle\widetilde B(P_0a,P_0c),P_0b\rangle
=-\langle B_0(a,c),b\rangle.
\]
\end{proof}

\begin{theorem}[Invariant harmonic fibres and fibre-wise attractors]\label{thm:fibre}
Consider \eqref{eq:u}--\eqref{eq:h} with $B=B_0$ from \eqref{eq:decoupledB} and assume $QPf=QPg=0$. For every $(u_H,h_H)\in\mathcal H^1\times\mathcal H^1$, the affine fibre
\[
\mathcal X_{u_H,h_H}:=\{(u,h)\in\mathcal V\times\mathcal V:Qu=u_H,\ Qh=h_H\}
\]
is positively invariant. On this fibre, the complementary variables $(P_0u,P_0h)$ satisfy the energy identity \eqref{eq:energyidentity} with $\Delta_1|_{\mathcal V_0}$ and forcings $P_0Pf,P_0Pg\in\mathcal V_0$; under the displayed assumption these equal $Pf,Pg$. Hence the fibre system has a compact global attractor. In the unforced case,
\[
\|P_0u(t)\|^2+\|P_0h(t)\|^2
\leq
\bigl(\|P_0u(0)\|^2+\|P_0h(0)\|^2\bigr)e^{-2\beta_+t},
\]
where $\beta_+:=\min(\nu,\eta)\lambda_+$. Thus every unforced trajectory on $\mathcal X_{u_H,h_H}$ converges to $(u_H,h_H)$.
\end{theorem}

\begin{proof}
Lemma~\ref{lem:decoupled} and $QPf=QPg=0$ make the right-hand sides of \eqref{eq:harmonicu}--\eqref{eq:harmonich} zero. Hence the fibre is invariant. Write $u=u_H+u_0$ and $h=h_H+h_0$, with $u_0,h_0\in\mathcal V_0$. Since $B_0$ depends only on $P_0a$, the complementary variables solve a closed system with interaction $B_0(u_0,\cdot)$ and forcings $P_0Pf,P_0Pg$. The proof of Theorem~\ref{thm:energy} applies to $E_0=\|u_0\|^2+\|h_0\|^2$, and Theorem~\ref{thm:coercivity} supplies the constant $\lambda_+$ on $\mathcal V_0$. The forced absorbing-ball and omega-limit argument then applies relative to the fixed fibre. The unforced estimate follows without Young's inequality.
\end{proof}

Theorem~\ref{thm:fibre} does not assert that the original interaction in Proposition~\ref{prop:toy} has invariant harmonic fibres. Instead, it characterizes a concrete sufficient structure. This distinction is essential: topology obstructs full-space coercivity for the general cancellation class, while harmonic decoupling restores a well-posed dissipative theory on each prescribed harmonic fibre.

\subsection{What the fibre theorem does and does not resolve}

The theorem isolates a genuine mechanism rather than concealing the noncoercive difficulty. In the general system, the right-hand sides of \eqref{eq:harmonicu}--\eqref{eq:harmonich} contain quadratic terms that can change both the harmonic coordinates and the complementary energy. Thus the unrestricted system may exhibit harmonic exchange, as demonstrated below on a two-hole complex. The theorem does not claim that all admissible interactions have bounded full-space attractors in this case.

By contrast, $B_0$ removes every nonlinear input to the harmonic equation and every dependence of the complementary equation on the fixed harmonic coordinates. The restricted system is therefore not merely a projected simulation: it is a well-defined model on each affine fibre. In the unforced case its attractor is the single harmonic state $(u_H,h_H)$ in that fibre. In the forced case, the complementary forcing produces a compact fibre-wise attractor, while the prescribed harmonic state remains unchanged. This is the precise sense in which harmonic directions can be controlled without introducing an artificial damping term.

A different extension would add harmonic damping terms $-\gamma Qu$ and $-\gamma Qh$ with $\gamma>0$. Such a model restores a coercive full-space linear operator and is suitable when decay of harmonic circulation is intended. It is not adopted here because it changes the dynamics rather than characterizing the topology-induced neutral directions of the original Hodge-Laplacian system.

\subsection{Reduced equations on an affine fibre}

The fibre construction can be written in coordinates that make its spectral content transparent. Fix $(u_H,h_H)\in\mathcal H^1\times\mathcal H^1$ and write
\[
u=u_H+u_0,
\qquad
h=h_H+h_0,
\qquad
u_0,h_0\in\mathcal V_0.
\]
For the interaction $B_0$, the equations for $(u_0,h_0)$ are
\begin{align}
\dot u_0&=-\nu\Delta_1u_0+B_0(u_0,u_0)-B_0(h_0,h_0)+P_0Pf,\\
\dot h_0&=-\eta\Delta_1h_0+B_0(u_0,h_0)-B_0(h_0,u_0)+P_0Pg.
\label{eq:fibrereduction}
\end{align}
No coefficient in the two reduced equations \eqref{eq:fibrereduction} depends on $(u_H,h_H)$. Hence all unforced harmonic fibres carry isomorphic complementary dynamics, translated by their fixed harmonic labels. In particular, the linearization at the unforced fibre equilibrium $(u_H,h_H)$ is
\[
\begin{pmatrix}-\nu\Delta_1|_{\mathcal V_0}&0\\0&-\eta\Delta_1|_{\mathcal V_0}\end{pmatrix}.
\]
Its spectrum lies in $(-\infty,-\beta_+]$, with $\beta_+=\min(\nu,\eta)\lambda_+$. Although the nonlinear terms modify the modal dynamics, their exact energy cancellation allows the linear coercive decay estimate to extend globally to the nonlinear complementary system.

This reduction also explains the numerical comparison in Figure~\ref{fig:fibres}. The left-hand panel is generated by the original admissible realization, for which the projected equations contain harmonic coupling terms. The right-hand panel is generated by \eqref{eq:fibrereduction}. The two runs therefore compare two mathematically distinct, explicitly specified systems rather than two parameter choices within one unexamined simulation.

\section{Deterministic numerical study}\label{sec:numerics}

All computations are deterministic. The complexes, orientations, matrices, initial conditions, forcings, and solver tolerances are fixed a priori; no random graph, random complex, random initial condition, or stochastic forcing is used. The numerical verification script uses SciPy's DOP853 method with relative tolerance $10^{-10}$ and absolute tolerance $10^{-12}$ \cite{Virtanen2020}.

The script follows the same algebraic order as the analysis. It first forms $d_0$ and $d_1$, verifies $d_1d_0=0$, and computes an orthonormal singular-value basis $H$ for $\ker d_0^*$. The projector is then $P=HH^*$, and the restricted Hodge matrix is $H^*\Delta_1H$. Its zero eigenspace defines the harmonic projector $Q$. This procedure gives the harmonic dimension, the positive spectral constant, and all projected initial fields without graph-dependent conventions that are left implicit.

For every run, the script evaluates three independent algebraic gates before time integration: skew-symmetry of the anticommutator matrix, trilinear antisymmetry of the interaction, and cancellation of the complete two-field nonlinear vector field. The harmonic-decoupled run additionally tests $QB_0(a,b)=0$. A failure of any gate above $10^{-10}$ raises an error rather than producing a figure. The numerical plots are therefore illustrations of a checked analytical structure, not evidence substituted for the proofs.

\subsection{Complexes and restricted spectra}

Three complexes are used. The triangulated disk is a six-sector fan with trivial first cohomology. The triangulated annulus has one harmonic direction. The third complex is a deterministic $6\times4$ triangulated rectangle with two separated unfilled squares. It has two harmonic directions and a nontrivial complementary space $\mathcal V_0$. Unlike an unfilled single cycle, whose divergence-free space can be entirely one-dimensional and harmonic, this two-hole example has enough room to display a meaningful distinction between harmonic and dissipative degrees of freedom. Figure~\ref{fig:complexes} gives the geometries, and Figure~\ref{fig:spectra} displays the restricted Hodge spectra. The red zero modes in the latter are exact topological kernel directions up to the stated numerical tolerance.

\begin{figure}
\centering
\includegraphics[width=\textwidth]{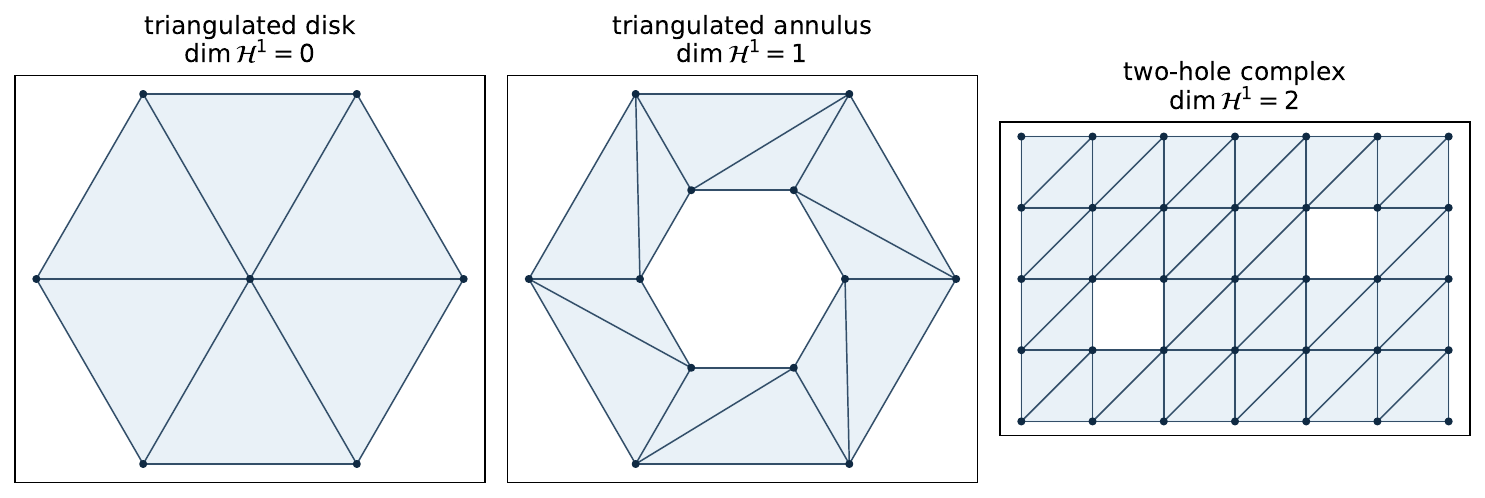}
\caption{Deterministic complex families. The disk has $\dim\mathcal H^1=0$, the annulus has $\dim\mathcal H^1=1$, and the two-hole complex has $\dim\mathcal H^1=2$.}
\label{fig:complexes}
\end{figure}

\begin{figure}
\centering
\includegraphics[width=\textwidth]{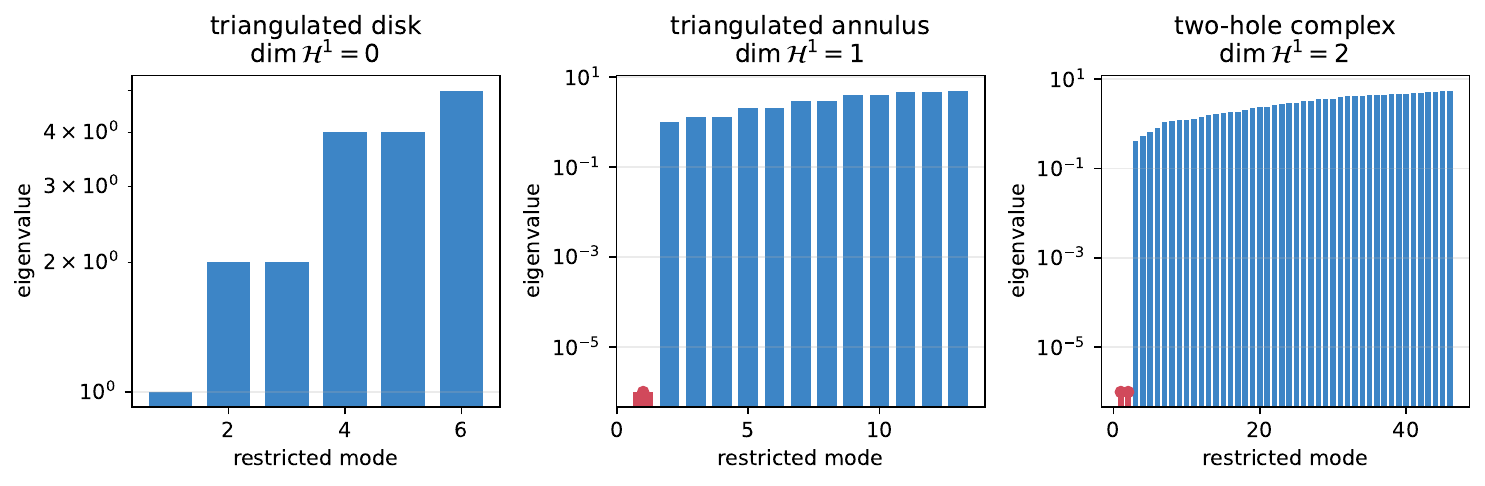}
\caption{Restricted Hodge spectra on the divergence-free spaces. Red markers identify harmonic zero modes. The disk is coercive, while the annulus and two-hole complex fail full-space coercivity.}
\label{fig:spectra}
\end{figure}

\subsection{Energy-law diagnostic}

For an edge space of dimension $m$, the fixed skew-symmetric matrix has nearest-neighbour entries
\[
J_{i,i+1}=0.22+0.03\bigl((3i)\bmod5\bigr),
\qquad
J_{i+1,i}=-J_{i,i+1},
\]
with closing entries $J_{0,m-1}=-0.17$ and $J_{m-1,0}=0.17$. The baseline parameters are $\nu=0.55$ and $\eta=0.35$. Projected trigonometric arrays provide deterministic initial states.

Before integration, the script checks the symmetry and idempotence of $P$, the relation $d_0^*P=0$, skew-symmetry of $M(a)$, trilinear antisymmetry, and total nonlinear energy cancellation. For the two-hole runs, the largest algebraic residual is below $1.34\times10^{-15}$. On the disk, Figure~\ref{fig:energy} compares unforced exponential decay with a bounded forced response. The right panel reports
\[
\frac{\left|E(t)-E(0)+2\int_0^t\bigl(\nu\|\Delta_1^{1/2}u(s)\|^2+
\eta\|\Delta_1^{1/2}h(s)\|^2\bigr)\,ds\right|}{E(0)}.
\]
The integral is evaluated by composite Simpson quadrature on 16,001 output times. The maximum normalized residual is $3.82\times10^{-10}$, making explicit that this is an a posteriori quadrature-and-solver diagnostic rather than the exact analytical cancellation statement of Theorem~\ref{thm:energy}.

\begin{figure}
\centering
\includegraphics[width=\textwidth]{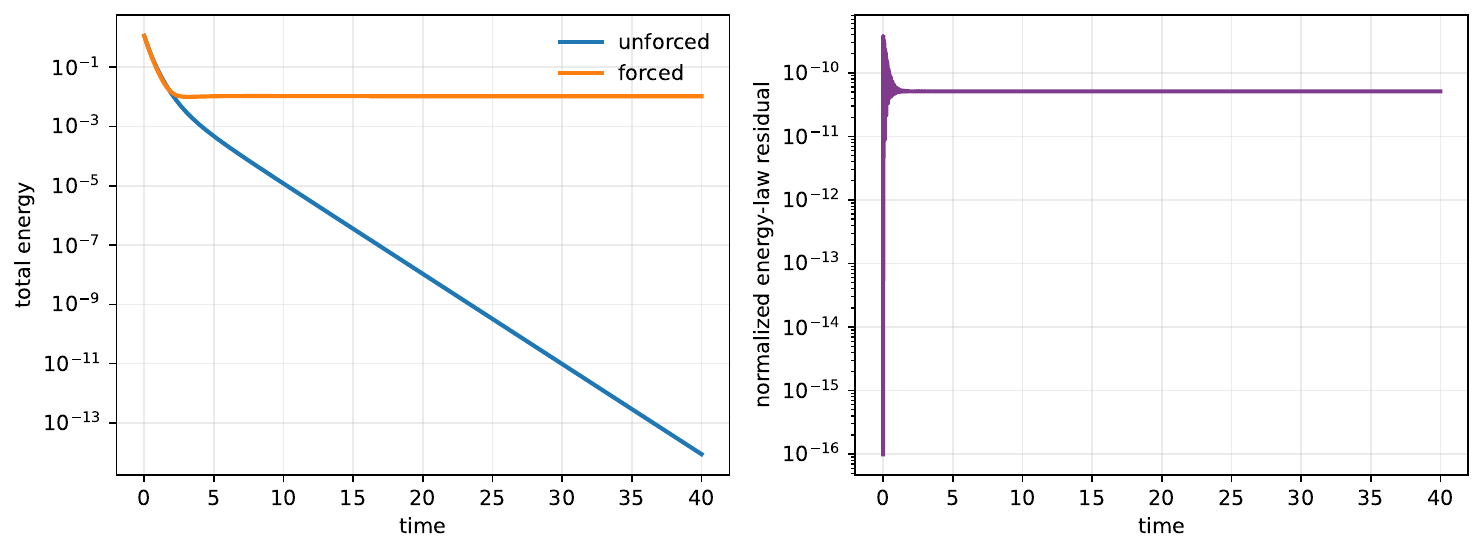}
\caption{Energy dynamics on the coercive disk. Left: unforced decay and a bounded deterministic forced response. Right: normalized residual of the integrated energy identity for the unforced run, using dense-output composite-Simpson quadrature.}
\label{fig:energy}
\end{figure}

\subsection{Two-hole harmonic fibres}

Figure~\ref{fig:fibres} addresses the noncoercive issue directly. In the left panel, the corrected interaction from Proposition~\ref{prop:toy} permits a nonzero change in harmonic energy on the two-hole complex; for the displayed deterministic trajectory its maximum change is $2.57\times10^{-3}$. This illustrates the transfer mechanism in \eqref{eq:harmonicu}--\eqref{eq:harmonich} without implying a universal residual-energy law. In the right panel, the harmonic-decoupled interaction from \eqref{eq:decoupledB} keeps the harmonic energy constant to a maximum change of $5.14\times10^{-16}$ while the complementary energy decays. The companion algebraic residual for harmonic decoupling is below $4.27\times10^{-16}$. Thus the computation distinguishes the general noncoercive case from the structured fibre-wise result of Theorem~\ref{thm:fibre}.

\begin{figure}
\centering
\includegraphics[width=\textwidth]{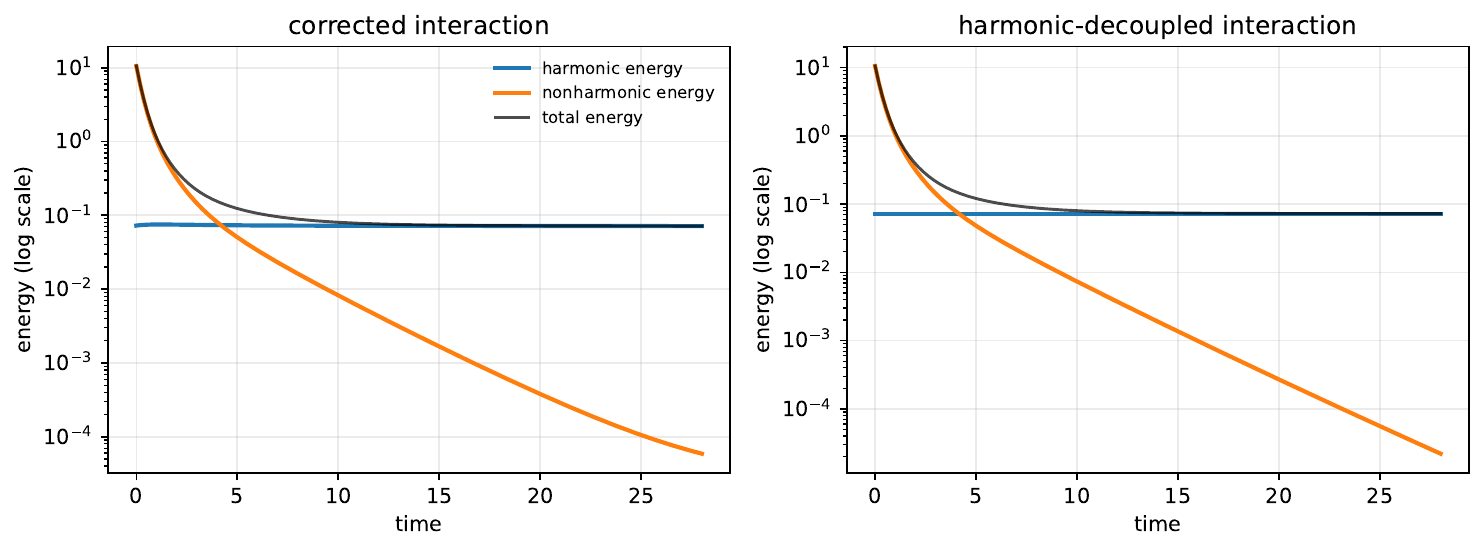}
\caption{Two-hole comparison on logarithmic energy scales. Left: the corrected interaction allows harmonic/nonharmonic energy exchange. Right: the harmonic-decoupled realization fixes the harmonic affine fibre while the complementary component dissipates.}
\label{fig:fibres}
\end{figure}

\section{Discussion}

The analysis separates three statements that should not be conflated. The trilinear cancellation identity supplies the energy law for a broad class of cell-complex edge-cochain interactions. The anticommutator $D(a)J+JD(a)$ is an explicit corrected realization of that class. The Hodge coercivity criterion then determines whether the standard full-space dissipative argument applies. In finite dimensions, the compact-attractor conclusion under coercivity is routine once an absorbing ball exists; the topology-sensitive availability of that ball is the distinctive result.

The noncoercive analysis makes the limitation constructive. General admissible interactions can transfer energy to and from harmonic cochains, so harmonic directions neither force nor exclude a nonzero limiting energy for every trajectory. The harmonic-decoupled construction identifies one sufficient mechanism that fixes the harmonic coordinates and yields attractors on affine fibres. Extending this theory to local cochain transport interactions, nonlinearities derived from a compatible discrete calculus, and quantitative dimension estimates remains open.

\appendix

\section{Orientation and numerical conventions}

Every oriented edge $(i,j)$ contributes a row to $d_0$ with entries $-1$ and $+1$ in the columns of $i$ and $j$. An oriented triangular face contributes a row to $d_1$ with sign determined by whether its boundary traversal agrees with the stored edge orientation. The identity $d_1d_0=0$ is then exact up to machine roundoff. An orthonormal basis $H$ of $\ker d_0^*$ is computed by singular-value decomposition, giving $P=HH^*$. The restricted Hodge matrix is $H^*\Delta_1H$; its zero eigenspace defines $Q$ and its positive spectrum supplies $\lambda_+$.

\section{Deterministic test specifications}

This appendix records the non-random choices used by the figure-generation script. For an edge space of dimension $m$, the initial data are
\[
u_0=0.62P\bigl(\cos(i+0.35)\bigr)_{i=0}^{m-1},
\qquad
h_0=0.46P\bigl(\sin(1.7i+0.20)\bigr)_{i=0}^{m-1}.
\]
For a run designed to display a harmonic fibre, a fixed harmonic bias is added to $u_0$ and the corresponding opposite-signed bias is added to $h_0$. The direction is obtained deterministically from a singular-vector calculation for $Q$ and normalized in the Euclidean edge norm. No random seed is used at any stage.

The skew matrix is sparse and fixed. Its adjacent entries are
\[
J_{i,i+1}=0.22+0.03((3i)\bmod5),
\qquad
J_{i+1,i}=-J_{i,i+1},
\]
with the closing pair $J_{0,m-1}=-0.17$ and $J_{m-1,0}=0.17$. The disk forcing, when used, is the projected pair
\[
f=0.065P\bigl(\cos(0.6i+0.1)\bigr)_{i=0}^{m-1},
\qquad
g=0.045P\bigl(\sin(0.9i+0.4)\bigr)_{i=0}^{m-1}.
\]
The script uses 16,001 equally spaced output times for the energy-law figure and 5,601 output times for the two-hole fibre diagnostic. The integration tolerances and the dense sampling have different roles: DOP853 controls the state approximation, while composite Simpson quadrature controls post-processing of the dissipation integral.

The algebraic validation threshold is $10^{-10}$. In the reported two-hole calculation, the maximum residual for the original anticommutator interaction is $1.34\times10^{-15}$. For the harmonic-decoupled realization, the maximum of the skew-symmetry, trilinear-antisymmetry, and $QB_0$ residuals is $4.27\times10^{-16}$. These values are stated to document deterministic numerical consistency; they are not used as assumptions in any theorem.

\section{Construction of the two-hole complex}

The two-hole example is a rectangular simplicial complex rather than a graph with a visually suggested hole. Its vertices are the lattice points
\[
(i,j),\qquad 0\leq i\leq6,\quad 0\leq j\leq4.
\]
Every unit square is divided along the same diagonal into two oriented triangles, except for the squares with lower-left corners $(1,1)$ and $(4,2)$, which are omitted. The two removed squares are separated by filled cells, and the resulting finite complex has two independent harmonic $1$-cochains. The construction is encoded directly from this face list, after which the script computes the harmonic dimension from the restricted Hodge spectrum rather than assuming it from a drawing.

This choice addresses the degeneracy of an unfilled single cycle. A one-dimensional divergence-free phase space is entirely harmonic and cannot distinguish exchange between harmonic and complementary directions. The two-hole complex has both types of directions. It therefore permits the direct comparison in Figure~\ref{fig:fibres}: a general admissible interaction can move energy between the two parts, whereas the harmonic-decoupled interaction fixes the harmonic affine fibre by construction.

\section{Algebraic verification protocol}

The numerical protocol checks identities in their matrix form before any ODE is solved. For a complex with edge dimension $m$, a representative projected triple $a,b,c\in\mathcal V\subset\mathbb R^m$ is formed from deterministic trigonometric coordinate arrays. The following residuals are computed:
\begin{align*}
&r_P=\max\{\|P^2-P\|,\|P-P^*\|,\|d_0^*P\|\},\\
&r_M(a)=\|M(a)+M(a)^*\|,\\
&r_b(a,b,c)=\bigl|\langle B(a,b),c\rangle+\langle B(a,c),b\rangle\bigr|,\\
&r_{\mathcal N}(u,h)=\bigl|\langle\mathcal N(u,h),(u,h)\rangle\bigr|.
\end{align*}
For the harmonic-decoupled realization, the additional range residual
\[
r_Q(a,b)=\|QB_0(a,b)\|
\]
is evaluated. The maximum of the applicable quantities must be below $10^{-10}$. This gate distinguishes exact algebraic structure, represented numerically up to floating-point roundoff, from the separate approximation error of a time integrator.

The energy-law plot uses a different check. Given an unforced numerical trajectory, it compares $E(t)-E(0)$ with the composite-Simpson approximation to the dissipation integral in \eqref{eq:integratedbalance}. This residual is expected to be larger than $r_M$, $r_b$, or $r_{\mathcal N}$ because it contains both state-approximation and quadrature effects. The revised dense-output calculation reduces the maximum normalized residual to $3.82\times10^{-10}$, while the algebraic residuals remain at the $10^{-15}$ scale. Reporting the two categories separately prevents an a posteriori quadrature diagnostic from being confused with the exact cancellation result.

\section*{Data Availability Statement}
No external datasets were used in this study. 

\section*{Conflict of interest}
The author declares no conflict of interest.

\bibliographystyle{plain}
\bibliography{mhd_gds_refs}

\end{document}